\documentclass{amsproc}

\usepackage{}
 \usepackage{relsize}

\usepackage{mathtools}

\newtheorem{theorem}{Theorem}[section]
\newtheorem{lemma}[theorem]{Lemma}
\newtheorem{corollary}[theorem]{Corollary}
\newtheorem{proposition}[theorem]{Proposition}

\theoremstyle{definition}

\newtheorem{example}[theorem]{Example}

\theoremstyle{remark}

\numberwithin{equation}{section}

\providecommand{\keywords}[1]{\textbf{\textit{Keywords:}} #1}

\begin{document}

\title[The polynomial CVP]{The polynomial cluster value problem}


\author{Sof\'ia Ortega Castillo}
\address{Centro de Investigaci\'on en Matem\'aticas A.~C.,
Jalisco S/N, Col. Valenciana, 36023, Guanajuato, Guanajuato, Mexico.}
\curraddr{}
\email{sofia.ortega@cimat.mx}
\thanks{The first author was partially supported by an IMU/CWM supplementary grant, and would like to thank Universidad Complutense de Madrid for welcoming support on short visits.}

\author{\'Angeles Prieto 
}
\address{Departamento de An\'alisis Matem\'atico, Facultad de CC. Matem\'aticas,
Universidad Complutense de Madrid, Plaza de Ciencias, 3, 28040, Madrid, Spain.}
\curraddr{}
\email{angelin@mat.ucm.es}
\thanks{The second author was partially supported by MINECO, grants MTM2012-34341 and MTM2015-65825-P, and would like to thank the politeness and support received at CIMAT}


\maketitle

\begin{center}
January 24th, 2018
\end{center}

\begin{abstract} The polynomial cluster value problem replaces the role of the continuous linear functionals in the original cluster value problem for the continuous polynomials to describe the corresponding cluster sets and fibers. We prove several polynomial cluster value theorems for uniform algebras $H(B)$ between $A_u(B)$ and $H^{\infty}(B)$, where $B$ is the open unit ball of a complex Banach space $X$. We also obtain new results about the original cluster value problem, especially for $A_{\infty}(B)$. Examples of spaces $X$ considered here are spaces of continuous functions, $\ell_1$ and locally uniformly convex spaces.
\end{abstract}



\section{Introduction}

The original cluster value problem concerns the boundary behavior of complex-valued, bounded, holomorphic functions on the open unit ball $B$ of a complex Banach space $X$. The space of bounded holomorphic functions forms the unital, commutative Banach algebra $H^{\infty}(B)$. This is, however, a very large algebra, so it is natural to work with its more manageable unital subalgebras $H(B)$ containing the bounded linear functionals $X^*$. For example, $A(B)$ is the subalgebra of $H^{\infty}(B)$ generated precisely by $X^{*}|_{B}$ and the constant $1$.  


\medskip

For a commutative Banach algebra $H(B)$ as before, there is a norm-decreasing homomorphism from $H(B)$ to the continuous functions on its spectrum, called the Gelfand transform. This homomorphism results an isometric embedding since the norm of this Banach algebra equals its spectral radius. As it turns out, the cluster value problem can be generally posed in terms of the spectrum of $H(B)$. This spectrum consists of the nonzero multiplicative linear functionals on $H(B)$, denoted $M_{H(B)}$. Unital and commutative Banach algebras $H(B)$ have its spectrum contained in the norm-one linear functionals $S_{H(B)^*}$, and this spectrum is a nonempty compact space in the usual weak-star topology making it a subspace of the unit ball of $H(B)^*$. Obvious members in the spectrum of $H(B)$, also called its characters, are evaluations at points of $B$. Other characters can be found through its respective kernel, e.~g. when it is the maximal ideal containing the functions that nullify over a certain limit \cite[\S 10.1]{Hf}, \cite[\S 7]{ACG}, \cite[Pr.~1.5]{AGGM}.

\medskip

Identifying the whole spectrum seems out of reach, but there have been advances discovering its structure. For example, the spectrum of $H^{\infty}$ of the unit disk $\mathbb{D}$ can be naturally mapped to the closed unit disk by evaluating each character at the identity function $z \mapsto z$. It is known that this mapping $\pi$ is continuous and surjective, and it is injective on the inverse image of the open unit disk. The remainder of the spectrum is mapped to the unit circle, so for each $\alpha  \in \mathbb{C}$ with $|\alpha|=1$ we call $\pi^{-1}(\alpha)$ the fiber over $\alpha$ of the spectrum. In 1961 I.~J.~Schark discovered that, for each $f$ in $H^{\infty}$ of the unit disk $\mathbb{D}$, the range of the Gelfand transform $\hat{f}$ on the fiber over any unitary $\alpha \in \mathbb{C}$ coincides with the cluster set of $f$ at $x_0$, i.e. the limit values of $f(x)$ as $x \in \mathbb{D}$ tends to $x_0$ in norm \cite{IJS}. 


\medskip

The last result was surpassed in 1962 by the Corona theorem of Carleson. The Corona of the spectrum of a Banach algebra $H(B)$ consists of those characters that are not in the closure of the ball $B$. The Corona problem asks whether such Corona is empty and, if so, we say we have a Corona theorem. Carleson  proved that the Corona of $H^{\infty}(\mathbb{D})$ is empty \cite{C}. Fundamental results about the Corona problem appear in \cite{DKSTW}.

\medskip

In the next decade, respective cluster set identifications were established despite the corresponding Corona staying unknown: by Gamelin in 1973 for the algebra $H^{\infty}$ of the polydisk in $\mathbb{C}^n$ \cite{G-itfa}, and by McDonald in 1979 for $H^{\infty}$ of the Euclidean unit ball in $\mathbb{C}^n$ or any other strongly pseudoconvex domain in $\mathbb{C}^n$ with smooth boundary \cite{Mc}. The latter result partially relied on the existence of peaking functions at points of strong pseudoconvexity \cite{Ro}.

\medskip

Passing from $\mathbb{C}^n$ to an arbitrary ambient space $X$, the cluster set of $f \in H(B)$ at $x^{**}\in \bar{B}^{**}$ consists of the accumulation points of values $f(x)$ as $x\in B$ tends to $x^{**}$ in the weak-star topology. Note that the weak-star topology is the initial topology with respect to $X^*$, and given that $H(B)$ contains the algebra generated by $X^*$ and $1$, we consider the mapping $\pi$ restricting each character to $X^*$ to become a character of $A(B)$, and again partition the spectrum of $H(B)$ in fibers via inverse images of $\pi$, ending up with each fiber associated to a point of the closed ball of the bidual $\bar{B}^{**}$. When the Corona of $H(B)$ is empty, once again the boundary behavior of each $f \in H(B)$ is determined: the cluster set of $f \in H(B)$ at every $x^{**}\in \bar{B}^{**}$ coincides with the Gelfand transform $\hat{f}$ evaluated in the fiber over $x^{**}$ \cite{ACGLM}. 

\medskip

A cluster value theorem for the Banach algebra $H(B)$ at the point $x^{**} \in \bar{B}^{**}$ asserts that the previously described cluster set identification holds (even if we do not know whether the Corona is empty). If this cluster set characterization is moreover true for all $x^{**} \in \bar{B}^{**}$ then there is a cluster value theorem for $H(B)$. The cluster value problem was first set up in this generality in \cite{ACGLM}. Several broad ideas about the cluster value problem were first established in \cite{ACGLM}, \cite{JO2} and \cite{ACLM}, and the first infinite-dimensional cluster value theorems were proved in \cite{ACGLM} and \cite{JO} focusing on Banach algebras on the ball of some classical Banach spaces. Examples include the uniformly continuous and holomorphic functions $A_u(B)$ on the ball $B$ of any Hilbert space, and $H^{\infty}$ for the ball of $c_0$ or any space of continuous functions on a scattered compact Hausdorff space. The state of the art on the cluster value problem was surveyed in \cite{CrGrMsSP}. 

\medskip

It is worth mentioning that for finite-dimensional spaces, $c_0$ and the spaces of continuous functions on scattered compact Hausdorff spaces $K$ the algebra $A_u(B)$ coincides with the ball algebra $A(B)$. In contrast, as soon as $K$ is not scattered, the algebras $A(B)$ and $A_u(B)$ are no longer equal \cite{JO}. Since $A_u(B)$ is the subalgebra of $H^{\infty}(B)$ generated by the continuous polynomials on $X$, $P(X)$ \cite[p.~56]{ACG}, while $A(B)$ is generated by $X^{*}|_{B}$ and $1$, it was first considered in \cite[p.~1565]{JO} to specialize in the \textit{cluster value problem for $H^{\infty}(B)$ over $A_u(B)$} in which the weak-star topology is replaced by the polynomial-star topology (to be described in Section 2) and the mapping $\pi$ is replaced by the restriction $\pi^{\mathcal{P}}$ of characters to $A_u(B)$.  In this paper we analyze the polynomial cluster value problem just described but for any algebra $H(B)$ between $A_u(B)$ and $H^{\infty}(B)$, that still allows considering the restriction mapping $\pi^{\mathcal{P}}$  from the characters of $H(B)$ into the characters of $A_u(B)$. The spectrum of $A_u(B)$ is studied in some detail in \cite{ACG}. As the polynomial-star topology is halfway the norm topology and the weak-star topology and it still preserves the density of the ball in the closed ball of the bidual \cite{DG}, this topology gives rise to a finer cluster value problem for which the polynomial cluster sets are contained in the original cluster sets.


\medskip

The Banach algebras $H(B)$ in this paper are uniform algebras, which are natural settings to study holomorphy and boundary value questions. A uniform algebra on a compact set $K$ is known to be a closed, separating subalgebra of $C(K)$ containing the constants. In general, a uniform algebra $H$ is a complex Banach algebra with unit satisfying that $\lim_{n\to\infty}\|f^n\|^{1/n}=\|f\|$ for all $f \in H$, as in such case it is a separating subalgebra of the continuous functions on its spectrum \cite{B}. We have mentioned the examples $A(B) \subset A_u(B) \subset H^{\infty}(B)$, and between $A_u(B)$ and $H^{\infty}(B)$ is the algebra $A_{\infty}(B)$ of bounded holomorphic functions on $B$ that extend continuously to the boundary. The inclusion of $A_u(B)$ in $A_{\infty}(B)$ is proper \cite[\S 12]{ACG} so we can also consider any intermediate algebra generated by $A_u(B)$ and finitely many arbitrary functions in $A_{\infty}(B)\setminus A_u(B)$. Among $A_{\infty}(B)$ and $H^{\infty}(B)$ we find the intermediate algebras of bounded holomorphic functions that extend continuously at finitely many points in $\bar{B}$.


\medskip

The work outline is the following. We begin Section 2 by proving general results about the polynomial cluster value problem for uniform algebras between $A_u(B)$ and $H^{\infty}(B)$. Such is an algebraic formulation of the polynomial cluster value problem that resembles a well known algebraic form of the Corona problem. Our aim is to display evidence that the polynomial cluster value problem is a natural analogue to the original cluster value problem.

\medskip

In Section 3 we focus on the polynomial cluster value problem for $H^{\infty}(B_{C(K)})$, proving that we can reduce it to checking the behavior at $0$ and at points in the unit sphere of $C(K)^{**}$. This strengthens the results in \cite[pp.~1567-1568]{JO}.

\medskip

In Section 4 we discuss peak points and strong peak points for complex-valued function spaces over a metric space. We mention examples and counterexamples of such kind of points. In particular, we refine known methods to show that a point where a given Banach space is locally uniformly convex is a strong peak point for $A(B)$. We use those examples, along with further results, to prove polynomial and original cluster value theorems for $A_{\infty}(B)$.


\medskip

For background on commutative Banach algebras and uniform algebras, see \cite{G-ua}. The reader may also want to learn about infinite-dimensional holomorphy from \cite{Mu} or through the selected results in \cite[\S 1.4]{O}.

\section{General results}

The notion of cluster set for a holomorphic function $f$ and an element $x^{**} \in X^{**}$ is provided in \cite{ACGLM} in terms of the weak-star topology. Indeed, we consider all possible limits of $f(x_{\alpha})$ for nets $(x_{\alpha})\subset B$ weak-star converging to $x^{**}$, i.e. $x_{\alpha}(x^*)\to x^{**}(x^*)$ for each linear functional $x^{*} \in X^{*}$. And the fibers are given in terms of $X^*$ as well, since for each $x^{**}\in \bar{B}^{**}$ the fiber over $x^{**}$ is $M_{x^{**}}:=\pi^{-1}(\{x^{**}\})$ where $\pi$ is the following homomorphism among spectra of algebras
\begin{align*}
\pi: M_{H(B)} & \to M_{A(B)}=\bar{B}^{**},\\
\tau & \mapsto \tau|_{X^*}.
\end{align*}
Note that for $x \in B$ the character $\delta_x \in M_{H(B)}$ which evaluates at the point $x$ is in the fiber $M_x$. The compactness of $M_{H(B)}$ and the continuity of $\pi$ in turn assure that $\pi$ is onto, namely each fiber is nonempty.

\medskip

Now we will describe the polynomial cluster sets and fibers in terms of the polynomial-star topology and $P(X)$. It is known that each polynomial $P \in P(X)$ admits a canonical extension $\tilde{P} \in P(X^{**})$ which is called its Aron-Berner extension \cite[Section 2]{G-afbs}. This canonical extension is obtained based on the Arens extension of the symmetric mapping $A$ associated to the polynomial $P$. The Arens extension is obtained extending by weak-star continuity, one variable at a time, each variable of $A$. The resulting mapping $\tilde{A}$ depends on the order the variables were extended and might not be symmetric, but the $n!$ possible extensions coincide on the diagonal. This common restriction defines $\tilde{P}$. The polynomial-star topology in the second dual of $X$ is the smallest topology that makes the Aron-Berner extension $\tilde{P}$ of every polynomial $P$ on $X$ up to $X^{**}$ continuous; this topology can be denoted by $w(X^{**}, P(X))$. The nets $(x_{\alpha})$ converging in the topology $w(X^{**}, P(X))$ to $x^{**}$ are those such that $\tilde{P}(x_{\alpha})\to \tilde{P}(x^{**})$ for all $P \in P(X)$. More basic results about $w(X^{**}, P(X))$ can be found in \cite[Section 2.1]{G-afbs}. The notation used here for the Aron-Berner extension is not standard, but it will help us to avoid confusion with the Gelfand transform and with complex conjugation.

\medskip

Throughout this section we will refer to $H(B)$ as any uniform algebra between $A_u(B)$ and $H^{\infty}(B)$. Given a holomorphic function $f \in H(B)$ as well as $x_0^{**} \in \bar{B}^{**}$, the {\it polynomial cluster set} of $f$ at $x_0^{**}$ is the nonempty set of accumulation points of values $f(x)$ as $x \in B$ tends to $x_0^{**}$ in the polynomial-star topology, i.e.
$$Cl_B^{\mathcal{P}}(f, x_0^{**}):=\{\lambda: \exists (x_{\alpha})\subset B \; | P(x_{\alpha})\to \tilde{P}(x_0^{**}) \; \text{ for all } P \in P(X), \; f(x_{\alpha})\to\lambda\}.$$

\medskip

 Given $x_0^{**} \in \bar{B}^{**}$, $\delta_{x_0^{**}} \in M_{A_u(B)}$ evaluates the Aron-Berner extension of $g \in A_u(B)$ at $x_0^{**}$, which is possible due to Lemma \ref{lm2.1} below and the continuous extension to the boundary of $\tilde{g} \in A_u(B^{**})$. The {\it polynomial fiber} of $M_{H(B)}$ at $x_0^{**} \in \bar{B}^{**}$ is the inverse image of $\delta_{x_0^{**}}$ under the restriction map $\pi^{\mathcal{P}} : M_{H(B)} \to M_{A_u(B)}$. We do not know the range of $\pi^{\mathcal{P}}$ but each fiber at a point in $\bar{B}^{**}$ is still nonempty as $B$ is dense in $\bar{B}^{**}$ in the polynomial-star topology \cite[Thm.~2]{DG}.



\medskip

 The polynomial cluster value problem can be posed as follows: if we denote the polynomial fiber of $M_{H(B)}$ at $x_0^{**}$ by $M_{x_0^{**}}^{\mathcal{P}}(H(B))$ (and if there is no confusion, we write $M_{x_0^{**}}^{\mathcal{P}}(B)$ or $M_{x_0^{**}}^{\mathcal{P}}$), and denoting the Gelfand transform of $f \in H(B)$ by $\hat{f}$, is it true that
\begin{equation}\label{eq2.1}
Cl_B^{\mathcal{P}}(f, x_0^{**}) = \hat{f}(M_{x_0^{**}}^{\mathcal{P}}), \; \; \text{ for all } f \in H(B) ?
\end{equation}

We have a \textit{polynomial cluster value theorem for $H(B)$ at $x_0^{**}$} when the equation \ref{eq2.1} holds.

\medskip

As in the case of the cluster value problem, an inclusion in the polynomial cluster value problem is always true. We use the following folklore result for that. 

\begin{lemma}\label{lm2.1}
The Aron-Berner extension of an element in $A_u(B)$ is in $A_u(B^{**})$.
\end{lemma}

\begin{theorem}\label{thm2.2}
For every $x_0^{**} \in \bar{B}^{**}$ and $f \in H(B)$,
\begin{equation}\label{eq2.3}
Cl_B^{\mathcal{P}}(f, x_0^{**})\subset \hat{f}(M_{x_0^{**}}^{\mathcal{P}}).
\end{equation}
\end{theorem}
\begin{proof}
Suppose that $\lambda \in Cl_B^{\mathcal{P}}(f, x_0^{**})$. Then there exists a net $(x_{\alpha}) \subset B$ such that $f(x_{\alpha})\to \lambda$ and $x_{\alpha} \to x_0^{**}$ in the topology $w(X^{**}, P(X))$. Without loss of generality we may assume that $x_{\alpha} \to \tau \in M_{H(B)}$, because the spectrum $M_{H(B)}$ is compact. Then $\hat{f}(\tau)=\lim_{\alpha} f(x_{\alpha})=\lambda$. Moreover, for every $g \in A_u(B)$, $$\hat{g}(\tau)=\lim_{\alpha} g(x_{\alpha})=\tilde{g}(x_0^{**}),$$ where the second equality holds because the polynomials on $X$ are uniformly dense in $A_u(B)$, and because $\tilde{g} \in A_u(B^{**})$. Thus $\tau \in M_{x_0^{**}}^{\mathcal{P}}$, so $\lambda \in \hat{f}(M_{x_0^{**}}^{\mathcal{P}})$. 
\end{proof}

As a consequence of Theorem \ref{thm2.2}, one can see that the larger the fiber $M_{x_0^{**}}^{\mathcal{P}}$, the harder it is to get a cluster value theorem at $x_0^{**}$. And the smaller the images of such fiber under the Gelfand transforms of functions $f \in H(B)$, the easier it is to obtain the desired theorem. The following consequence is intuitively clear too.

\begin{corollary}
Let $x_0^{**} \in \bar{B}^{**}$ and $f\in H(B)$. If $\hat{f}$ is constant in $M_{x_0^{**}}^{\mathcal{P}}$, then $f$ extends to $x_0^{**}$ with $w(X^{**}, P(X))$-continuity. That is, if $\hat{f}(M_{x_0^{**}}^{\mathcal{P}})=\{\lambda\}$, then every net $(x_{\alpha})$
$w(X^{**}, P(X))$-convergent to $x_0^{**}$ satisfies that $\lim f(x_{\alpha})=\lambda$.
\end{corollary}
\begin{proof}
Let us see that if $(x_{\alpha})$ converges to $x_0^{**}$ in the topology $w(X^{**}, P(X))$, it has a subnet $(x_{\alpha_i})$ such that $\lim_i f(x_{\alpha_i})=\lambda$. Since $(f(x_{\alpha}))$ is bounded, there is a subnet $(x_{\alpha_i})$ such that $(f(x_{\alpha_i}))_i$ converges to a certain $\beta \in \mathbb{C}$. Then Theorem \ref{thm2.2} implies that there exists $\psi \in M_{x_0^{**}}^{\mathcal{P}}$ such that $\hat{f}(\psi)=\beta$. However, $\hat{f}(\psi) \in \hat{f}(M_{x_0^{**}}^{\mathcal{P}})=\{\lambda\}$. Therefore $\hat{f}(\psi)=\lambda$ and thus $f$ extends to $x_0^{**}$ with $w(X^{**}, P(X))$-continuity.
\end{proof}

In the same way we can prove the following result related to the original cluster value problem. Recall that weak-star continuous functions are always norm-continuous. 


\begin{proposition}
If $x_0^{**} \in \bar{B}^{**}$ and $f\in H(B)$ satifies that $\hat{f}$ is constant in $M_{x_0^{**}}$, then $f$ extends to a norm-continuous function $f_0$ on $B\cup \{x_0^{**}\}$ which is $w(X^{**}, X^*)$-continuous at $x_0^{**}$.
\end{proposition}

Examples of more comparisons between the cluster value problem and the polynomial cluster value problem are the following. First, if $H(B)$ is a uniform algebra between $A_u(B)$ and $H^{\infty}(B)$, the cluster value problem for $H(B)$ coincides with the polynomial cluster value problem for the same algebra when $X$ is a finite-dimensional Banach space. The reason is that in that case $A_u(B)=A(B)$, and hence the polynomial-star topology coincides with the weak-star topology in $X^{**}=X$. In general, whenever $A_u(B)=A(B)$, we have that the cluster value problem for $H(B)$ coincides with the polynomial cluster value problem for the same algebra. As already mentioned, this is the case for spaces of continuous functions $C(K)$ with $K$ compact, Hausdorff and scattered, and for $c_0$. It is meanwhile shown in \cite[p.~58]{ACG} that this is the case for the dual of the Tsirelson space.

\medskip

In view of Theorem \ref{thm2.2}, to prove the polynomial cluster value theorem for $H(B)$ at $x_0^{**}$ it is enough to show the reverse inclusion to the one in Equation \ref{eq2.3}, for all $f\in H(B)$. Or we can show the algebraic form of the polynomial cluster value problem that we shall present shortly in Theorem \ref{a-pcvp}. Let us first see the following algebraic version of the original cluster value problem that improves \cite[Lemma 2.3]{ACGLM}. 

\begin{theorem}\label{a-cvp}
There is a cluster value theorem for $H(B)$ at $x^{**} \in \bar{B}^{**}$ if and only if, for every finite family $f_1, \cdots, f_{n-1} \in A(B)$ and $f_n \in H(B)$ such that there exists $\delta>0$ for which $|f_1(x)|+\cdots+|f_n(x)|\geq\delta$ for all $x \in B$, it holds that $\hat{f_1}, \cdots, \hat{f_n}$ have no common zeroes in $M_{x^{**}}$.
\end{theorem}
\begin{proof}
$\Rightarrow$) If $\hat{f_1}, \cdots, \hat{f_n}$ had a common zero in $M_{x^{**}}$, there would be $\phi \in M_{x^{**}}$ such that $\hat{f_j}(\phi)=0$, $1\leq j \leq n$. In particular, $f_j(x^{**})=\hat{f_j}(\phi)=0$, $1\leq j \leq n-1$. Also, since $\hat{f_n}(M_{x^{**}})=Cl_B(f_n, x^{**})$, we obtain that there exists $(x_{\alpha})\subset B$ such that $x_{\alpha}\xrightarrow{w^*}x^{**}$ and $f_n(x_{\alpha})\to0$. Thus $\inf_{x \in B} \sum_1^n |f_j(x)|=0$, a contradiction. 

\medskip

$\Leftarrow$) If the cluster value theorem for $H(B)$ fails at $x^{**}$, then, after substracting a constant if necessary, we can find $g \in H(B)$ such that $0 \in \hat{g}(M_{x^{**}})$ but $0 \notin Cl_B(g, x^{**})$. Let $\phi \in M_{x^{**}}$ be such that $\hat{g}(\phi)=0$. Also, let $c,\; \varepsilon>0$ and $L_1, \cdots, L_{n-1} \in X^*$ be such that if $x \in B$ and $|L_j(x)-x^{**}(L_j)|<\varepsilon$ for $1\leq j \leq n-1$ then $|g(x)|\geq c$. Taking $\delta=\min\{c, \varepsilon\}$ we obtain that $|g(x)|\geq \delta$ when $|L_j(x)-x^{**}(L_j)|<\delta$, $1\leq j \leq n-1$. Thus, taking $f_j=L_j-x^{**}(L_j)$ for $1\leq j \leq n-1$, and $f_n=g$, we get that $\sum_1^n |f_j(x)|\geq \delta$ for all $x \in B$. Nevertheless $\hat{f_j}(\phi)=\phi(L_j-x^{**}(L_j))=0$ for $1\leq j \leq n-1$, while $\hat{f_n}(\phi)=\hat{g}(\phi)=0$, i.e. $\phi$ is a common zero in $M_{x^{**}}$ of $\hat{f_1}, \cdots, \hat{f_n}$. 
\end{proof}

The corresponding characterization of the polynomial cluster value theorem replaces the role of $X^{*}$ by $P(X)$. The algebra $A(B)$ generated by $X^{*}$ is substituted by $A_u(B)$ which is generated by $P(X)$. And a polynomial fiber takes the place of the original fiber at the same point.

\begin{theorem}\label{a-pcvp}
There is a polynomial cluster value theorem for $H(B)$ at $x^{**} \in \bar{B}^{**}$ if and only if, for every finite family $f_1, \cdots, f_{n-1} \in A_u(B)$ and $f_n \in H(B)$ such that there exists $\delta>0$ for which $|f_1(x)|+\cdots+|f_n(x)|\geq\delta$ for all $x \in B$, it holds that $\hat{f_1}, \cdots, \hat{f_n}$ have no common zeroes in $M_{x^{**}}^{\mathcal{P}}$.
\end{theorem}
\begin{proof}
$\Rightarrow$) If $\hat{f_1}, \cdots, \hat{f_n}$ had a common zero in $M_{x^{**}}^{\mathcal{P}}$, there would be $\phi \in M_{x^{**}}^{\mathcal{P}}$ such that $\hat{f_j}(\phi)=0$, $1\leq j \leq n$. In particular, $f_j(x^{**})=\hat{f_j}(\phi)=0$, $1\leq j \leq n-1$. Also, since $\hat{f_n}(M_{x^{**}}^{\mathcal{P}})=Cl_B^{\mathcal{P}}(f_n, x^{**})$, we obtain that there exists $(x_{\alpha})\subset B$ such that $x_{\alpha}\to x^{**}$ in the polynomial-star topology and $f_n(x_{\alpha})\to 0$. Since the polynomials are dense in $A_u(B)$ \cite[p.~56]{ACG}, we get that $\inf_{x \in B} \sum_1^n |f_j(x)|=0$, a contradiction.

\medskip

$\Leftarrow$) If the polynomial cluster value theorem for $H(B)$ fails at $x^{**}$, then there exists $g \in H(B)$ such that $0 \in \hat{g}(M_{x^{**}}^{\mathcal{P}})$ but $0 \notin Cl_B^{\mathcal{P}}(g, x^{**})$. Let $\phi \in M_{x^{**}}^{\mathcal{P}}$ be such that $\hat{g}(\phi)=0$. Also, let $c,\; \varepsilon>0$ and $P_1, \cdots, P_{n-1} \in P(X)$ be such that if $x \in B$ and $|P_j(x)-\widetilde{P_j}(x^{**})|<\varepsilon$ for $1\leq j \leq n-1$ then $|g(x)|\geq c$. As before, we can take $\delta=\min\{c, \varepsilon\}$, $f_j=P_j-\widetilde{P_j}(x^{**})$ for $1\leq j \leq n-1$, and $f_n=g$, to get that $\sum_1^n |f_j(x)|\geq \delta$ for all $x \in B$. Still, $\hat{f_j}(\phi)=\phi(P_j-\widetilde{P_j}(x^{**}))=0$ for $1\leq j \leq n-1$, while $\hat{f_n}(\phi)=\hat{g}(\phi)=0$, i.e. $\phi$ is a common zero in $M_{x^{**}}^{\mathcal{P}}$ of $\hat{f_1}, \cdots, \hat{f_n}$. 
\end{proof}

\medskip

To finish this section, let us see the following direct proof of the fact that the Corona theorem for $H(B)$ implies the polynomial cluster value theorem for $H(B)$ at all points $x^{**} \in \bar{B}^{**}$.

\begin{theorem}
If $B$ is dense in $M_{H(B)}$, then for every $x_0^{**} \in \bar{B}^{**}$ and $f \in H(B)$, $Cl_B^{\mathcal{P}}(f, x_0^{**})= \hat{f}(M_{x_0^{**}}^{\mathcal{P}})$.
\end{theorem}
\begin{proof}
Let $x_0^{**} \in \bar{B}^{**}$ and $f \in H(B)$. Due to Theorem \ref{thm2.2}, it is enough to show that $\hat{f}(M_{x_0^{**}}^{\mathcal{P}}) \subset Cl_B^{\mathcal{P}}(f, x_0^{**})$. Given $\tau \in M_{x_0^{**}}^{\mathcal{P}}$, we can find $(x_{\alpha})\subset B$ such that $\delta_{x_{\alpha}}\to \tau$ in $M_{H(B)}$. In particular, since $A_u(B)\subset H(B)$, if $g \in A_u(B)$,
$$g(x_{\alpha}) \to \hat{g}(\tau)= \tilde{g}(x_0^{**}),$$
so $x_{\alpha} \to x_0^{**}$ in the topology $w(X^{**}, P(X))$. Since also $f(x_{\alpha})\to \hat{f}(\tau)$, we obtain that $\hat{f}(\tau) \in Cl_B^{\mathcal{P}}(f, x_0^{**})$. 
\end{proof}

\section{Polynomial cluster value problem for spaces of continuous functions}

Despite the commonalities between the original cluster value problem and the polynomial cluster value problem, we have found some features of the latter that up to now make it distinctive. In this section we focus on the polynomial cluster value problem for the algebra $H^{\infty}(B)$ for $B$ the ball of a space of continuous functions.


\medskip

For the original cluster value problem, the size and structure of fibers for some spectra of subalgebras of $H^{\infty}(B)$ have been recently studied in \cite{ACLM} and \cite{AFGM}. In contrast with the case of the original cluster value problem, we do not know if the union of the polynomial fibers  $M_{x_0^{**}}^{\mathcal{P}}$, with $x_0^{**} \in \bar{B}^{**}$, is all of $M_{H(B)}$. Still, the polynomial cluster value problem is in general nontrivial, as illustrated by the size of the polynomial fibers in the following example \cite[pp.~1565-1567]{JO}. 

\begin{example}[Johnson, Ortega Castillo]
If $K$ is an infinite compact Hausdorff space then each of the polynomial fibers $M_{f_0}^{\mathcal{P}}$, for $f_0 \in B_{C(K)}$ and the algebra $H^{\infty}(B_{C(K)})$,  contains a holomorphic copy of $B_{\ell_{\infty}}$.
\end{example}

Proposition \ref{pr3.3} below extends the last example to all polynomial fibers $M_{f_0^{**}}^{\mathcal{P}}$ with $f_0^{**} \in B_{C(K)^{**}}$. As in \cite[Lemma 3.5]{JO}, we use that $C(K)^{**}$ is a commutative $C^*$-algebra that extends the $C^*$ structure of $C(K)$ to obtain M\"obius type biholomorphisms in $B_{C(K)^{**}}$. Here we shall show two more properties about them for the extension, that they are Lipschitz and $(w(X^{**}, \mathcal{P}(X)),w(X^{**}, \mathcal{P}(X)))$-continuous. We will subsequently use such characteristics to prove Proposition \ref{pr3.3} and Corollary \ref{co3.5} in connection with the polynomial cluster value problem.

\medskip

In what follows, we will use functional calculus notation for the commutative space $C(K)^{**}$. For background on functional calculus in $C^*$-algebras, see \cite{Mr}. To be consistent with our previous notation, we will denote an element of $B_{C(K)^{**}}$ by $f^{**}$. The involution in $C(K)^{**}$ will be denoted by  $\text{ }\bar{}$, since it corresponds to complex conjugation for elements of $C(K)$.

\begin{lemma}\label{lm3.2}
For all $f_0^{**} \in B_{C(K)^{**}}$, the mapping $T_{f_0^{**}}: B_{C(K)^{**}} \to B_{C(K)^{**}}$ given by
\begin{equation}\label{eq3.1}
T_{f_0^{**}}(f^{**})=\frac{f^{**}-f_0^{**}}{1-\bar{f_0^{**}}f^{**}}
\end{equation}

is biholomorphic. Moreover, it is Lipschitz and $(w(X^{**}, \mathcal{P}(X)),w(X^{**}, \mathcal{P}(X)))$-continuous.
\end{lemma}
\begin{proof}


In equation \ref{eq3.1} above, we have that $\|\bar{f_0^{**}}f^{**}\|<1$, so $1-\bar{f_0^{**}}f^{**}$ is certainly invertible. The proof of $T_{f_0^{**}}$ being biholomorphic is analogous to \cite[Lemma 3.2]{JO}.

\medskip

To prove that $T_{f_0^{**}}$ is a Lipschitz function, note that for all $f^{**}, g^{**} \in B_{C(K)^{**}}$,
\begin{align*}
\|T_{f_0^{**}}(f^{**})-T_{f_0^{**}}(g^{**})\|&= \mathlarger{ \| } \frac{(1-|f_0^{**}|^2)(f^{**}-g^{**})}{(1-\bar{f_0^{**}} f^{**})(1-\bar{f_0^{**}} g^{**})} \mathlarger{ \| }. \\
\end{align*}

Since $0\leq 1-|f_0^{**}|^2\leq 1$ then $\| 1-|f_0^{**}|^2\| \leq 1$, so
$$
\|T_{f_0^{**}}(f^{**})-T_{f_0^{**}}(g^{**})\| \leq \| (1-\bar{f_0^{**}} f^{**})^{-1}\| \cdot \|(1-\bar{f_0^{**}} g^{**})^{-1} \| \cdot \|f^{**}-g^{**}\|.
$$

\medskip

Moreover, due again to the submultiplicativity of the norm,
\begin{align*}
\|(1-\bar{f_0^{**}} f^{**})^{-1}\|(1-\|f_0^{**}\|) 
&\leq \|(1-\bar{f_0^{**}} f^{**})^{-1}\|-\|\bar{f_0^{**}}(1-\bar{f_0^{**}} f^{**})^{-1}\|\\
&\leq \|(1-\bar{f_0^{**}} f^{**})^{-1}\|-\|\bar{f_0^{**}}f^{**}(1-\bar{f_0^{**}} f^{**})^{-1}\|\\
&\leq \| (1-\bar{f_0^{**}} f^{**})(1-\bar{f_0^{**}} f^{**})^{-1}\|\\
&\leq 1,
\end{align*}
and similarly $\|(1-\bar{f_0^{**}} g^{**})^{-1}\|(1-\|f_0^{**}\|)\leq 1$, therefore
$$\|T_{f_0^{**}}(f^{**})-T_{f_0^{**}}(g^{**})\| \leq \frac{\|f^{**}-g^{**}\|}{(1-\|f_0^{**}\|)^2},$$
i.e. $T_{f_0^{**}}$ is Lipschitz, and consequently uniformly continuous. 

\medskip

Given $P \in P(X)\subset A_u(B)$, using lemma \ref{lm2.1} we get that the Aron-Berner extension $\tilde{P} \in A_u(B^{**})$. Since $T_{f_0^{**}}$ is holomorphic and uniformly continuous we have that $\tilde{P} \circ T_{f_0^{**}} \in A_u(B^{**})$, so it is the uniform limit of polynomials. Hence $\tilde{P}\circ T_{f_0^{**}}$ is polynomial-star continuous, so given a net $(f_{\alpha}^{**})$ converging to $f^{**}$ in $(w(X^{**}, \mathcal{P}(X))$ we observe that
$$\tilde{P} (T_{f_0^{**}}(f_{\alpha}^{**}))\to \tilde{P} (T_{f_0^{**}}(f^{**})),$$
i.e. $T_{f_0^{**}}$ is $(w(X^{**}, \mathcal{P}(X)),w(X^{**}, \mathcal{P}(X)))$-continuous.
\end{proof}

Let us prove the property in Lemma \ref{lm3.4} below about the M\"obius type biholomorphisms just described, as we will use it to show Proposition \ref{pr3.3}.

\begin{lemma}\label{lm3.4}
For all $\psi \in H^{\infty}(B)$ and $f_0^{**} \in B_{C(K)^{**}}$, $$\widetilde{\mathit{(\tilde{\psi} \circ T_{f_0^{**}}|_{B_{C(K)}})}}=\tilde{\psi} \circ T_{f_0^{**}}$$

\end{lemma}
\begin{proof}
Using the Taylor series expansion of $\psi$ at $0$, and the way $\tilde{\psi}$ is defined, it is enough to assume that $\psi$ is an $m$-homogeneous polynomial $P_m$, with associated symmetric $m$-linear functional $A$. 

\medskip

Now consider the Taylor series expansion of $T_{f_0^{**}}$,
$$T_{f_0^{**}}(f^{**})=\sum_{n=0}^{\infty} g_n^{**} \cdot (f^{**})^n,$$
that we can obtain from manipulating equation \ref{eq3.1}. Moreover, we can use Cauchy's inequalities or functional calculus to get that $||g_n^{**}||\leq 1$ for all $n \in \mathbb{N}_0$.

\medskip

Then, because of the continuity of $\widetilde{P_m}$,
\begin{align*}
\widetilde{P_m}(T_{f_0^{**}}(f^{**}))
&=\widetilde{P_m}(\sum_{n=0}^{\infty} g_n^{**} \cdot (f^{**})^n)\\
&=\lim_{N\to\infty}\widetilde{P_m}(\sum_{n=0}^N g_n^{**} \cdot (f^{**})^n)\\
&=\lim_{N\to\infty} \sum_{0 \leq n_i \leq N} \tilde{A}(g_{n_1}^{**} \cdot (f^{**})^{n_1}, \cdots, g_{n_m}^{**} \cdot (f^{**})^{n_m})\\
&=\sum_{n_i \in \mathbb{N}_0}  \tilde{A}(g_{n_1}^{**} \cdot (f^{**})^{n_1}, \cdots, g_{n_m}^{**} \cdot (f^{**})^{n_m}),
\end{align*}
because we can bound the tail from $N+1$ to $\infty$ by $||A|| (\frac{||f^{**}||^{N+1}}{1-||f^{**}||})^m$, which clearly goes to $0$ as $N \to \infty$, pointwise.

\medskip

After rearranging the terms we obtain that 
$$\widetilde{P_m}(T_{f_0^{**}}(f^{**}))=\sum_{k=0}^{\infty} \sum_{n_1+\cdots+n_m=k} \tilde{A}(g_{n_1}^{**} \cdot (f^{**})^{n_1}, \cdots, g_{n_m}^{**} \cdot (f^{**})^{n_m}),$$

which is the Taylor series expansion of $\widetilde{P_m}\circ T_{f_0^{**}}$.

\medskip

Consequently the Taylor series expansion of $\widetilde{P_m}\circ T_{f_0^{**}}|_{B_{C(K)}}$ is
\begin{equation}\label{eq3.2}
\widetilde{P_m}\circ T_{f_0^{**}}|_{B_{C(K)}}(f)=\sum_{k=0}^{\infty} \sum_{n_1+\cdots+n_m=k} \tilde{A}(g_{n_1}^{**} \cdot (f)^{n_1}, \cdots, g_{n_m}^{**} \cdot (f)^{n_m}).
\end{equation}

Let us recall that $\tilde{A}$ is obtained through extending $A$ one variable at a time by weak-star continuity \cite[\S 2.1]{G-afbs}. The restriction of $\tilde{A}$ to the diagonal does not depend on the order the variables were picked since $A$ is symmetric. In turn, if each of the $k$-homogeneous polynomials in $f \in C(K)$ in the right-hand side of (\ref{eq3.2}) is weak-star continuously extended from the first to the last variable and then restricted to the diagonal, we clearly recover the same polynomials but evaluated in the bidual of $C(K)$. It is thus clear that $\widetilde{P_m}\circ T_{f_0^{**}}$ coincides with $\widetilde{\mathit{(\widetilde{P_m}\circ T_{f_0^{**}}|_{B_{C(K)}})}}$.
\end{proof}

\begin{proposition}\label{pr3.3}
The biholomorphism $T_{f_0^{**}}$ of $B_{C(K)^{**}}$ induces a mapping on $M_{H^{\infty}(B_{C(K)})}$ that maps the polynomial fiber $M_{x_0^{**}}^{\mathcal{P}}$ onto the polynomial fiber $M_{T_{f_0^{**}}(x_0^{**})}^{\mathcal{P}}$, for all $x_0^{**} \in B_{C(K)^{**}}$.
\end{proposition}
\begin{proof}
Since $T_{f_0^{**}}$ is holomorphic and uniformly continuous, by lemmas \ref{lm2.1} and \ref{lm3.2} we get that for all $\psi \in A_u(B_{C(K)})$ clearly $\tilde{\psi} \circ T_{f_0^{**}}|_{B_{C(K)}} \in A_u(B_{C(K)})$, where $\tilde{\psi}$ is the Aron-Berner extension of $\psi$.

\medskip


\medskip
Similarly, $\tilde{\psi} \circ T_{f_0^{**}}|_{B_{C(K)}} \in H^{\infty}(B_{C(K)})$ when $\psi \in H^{\infty}(B_{C(K)})$, since $\tilde{\psi} \in H^{\infty}(B_{C(K)^{**}})$ in that case \cite[Theorem 5]{DG}.

\medskip

Consequently the mapping $\widehat{T_{f_0^{**}}}:M_{H^{\infty}(B_{C(K)})} \to M_{H^{\infty}(B_{C(K)})}$ given by
$$\widehat{T_{f_0^{**}}}(\tau)(\psi)=\tau(\tilde{\psi} \circ T_{f_0^{**}}|_{B_{C(K)}})$$
is well-defined.

\medskip

To prove the desired properties of $\widehat{T_{f_0^{**}}}$, observe that given $\tau \in M_{x_0^{**}}^{\mathcal{P}}$ and $\psi \in A_u(B_{C(K)})$,
\begin{align*}
\widehat{T_{f_0^{**}}}(\tau)(\psi)=\tau(\tilde{\psi}\circ T_{f_0^{**}}|_{B_{C(K)}})=\tilde{\psi}(T_{f_0^{**}}(x_0^{**})),
\end{align*}
where the last equality follows from Lemma \ref{lm3.4} and the fact that $\tilde{\psi} \circ T_{f_0^{**}}|_{B_{C(K)}} \in A_u(B_{C(K)})$. Thus $\widehat{T_{f_0^{**}}}(\tau) \in M_{T_{f_0^{**}}(x_0^{**})}^{\mathcal{P}}$ for all $\tau \in M_{x_0^{**}}^{\mathcal{P}}$.

\medskip

Further, given $\tau \in M_{T_{f_0^{**}}(x_0^{**})}^{\mathcal{P}}$, we have that $\tilde{\tau} \in M_{H(B)}$ defined by 
$$\tilde{\tau}(\psi)=\tau(\tilde{\psi}\circ T_{-f_0^{**}}|_{B_{C(K)}})$$
is clearly in $M_{x_0^{**}}^{\mathcal{P}}$, and
\begin{align*}
\widehat{T_{f_0^{**}}}(\tilde{\tau})(\psi)=\tilde{\tau}(\tilde{\psi}\circ T_{f_0^{**}}|_{B_{C(K)}})=\tau(\psi)
\end{align*}
i.e. $\widehat{T_{f_0^{**}}}(\tilde{\tau})=\tau$.

\medskip

Thus $\widehat{T_{f_0^{**}}}$ maps the polynomial fiber $M_{x_0^{**}}^{\mathcal{P}}$ exactly to the polynomial fiber $M_{T_{f_0^{**}}(x_0^{**})}^{\mathcal{P}}$.
\end{proof}

It is worth mentioning that for the algebra $A_{\infty}(B_{C(K)})$, we can similarly map each polynomial fiber $M_{x_0}^{\mathcal{P}}$ onto the polynomial fiber $M_{T_{f_0}(x_0)}^{\mathcal{P}}$, for all $x_0, f_0 \in B_{C(K)}$ (since $T_{f_0}$ extends to a continuous map $\overline{T_{f_0}}:\overline{B_{C(K)}}\to\overline{B_{C(K)}}$, and thus for all $\psi \in A_{\infty}(B_{C(K)})$ clearly $\psi\circ T_{f_0} \in A_{\infty}(B_{C(K)})$, so we can proceed as in Proposition \ref{pr3.3}).

\begin{corollary}\label{co3.5}
For $X=C(K)$, the polynomial cluster value theorem for $H^{\infty}(B)$ at $0$ implies the polynomial cluster value theorem for $H^{\infty}(B)$ at every $f_0^{**} \in B_{C(K)^{**}}$.
\end{corollary}
\begin{proof}
Let $f_0^{**} \in B_{C(K)^{**}}$. Then for all $\psi \in H^{\infty}(B_{C(K)})$,
\begin{align*}
\widehat{\psi}(M_{f_0^{**}}^{\mathcal{P}})=\widehat{\psi}\circ \widehat{T_{-f_0^{**}}}(M_0^{\mathcal{P}})=(\tilde{\psi}\circ T_{-f_0^{**}}|_{B_{C(K)}})\widehat{\phantom{x}}(M_0^{\mathcal{P}})
\end{align*}
where the last equality holds since for all $\tau \in M_0^{\mathcal{P}}$,
\begin{align*}
\widehat{\psi}(\widehat{T_{-f_0^{**}}}(\tau))=\widehat{T_{-f_0^{**}}}(\tau)(\psi)=\tau(\tilde{\psi}\circ T_{-f_0^{**}}|_{B_{C(K)}}).
\end{align*}

Meanwhile,
\begin{align*}
Cl_B^{\mathcal{P}}(\tilde{\psi}\circ T_{-f_0^{**}}|_{B_{C(K)}},0)&=\{ \lambda: \exists (f_{\alpha})\subset B \; | \; f_{\alpha}\xrightarrow{w(X^{**}, \mathcal{P}(X))}0, \tilde{\psi}\circ T_{-f_0^{**}}(f_{\alpha})\to\lambda\}\\
&=\{ \lambda: \exists (g_{\alpha}^{**})\subset B^{**} \; | \; g_{\alpha}^{**}\xrightarrow{w(X^{**}, \mathcal{P}(X))}f_0^{**}, \tilde{\psi}(g_{\alpha}^{**})\to\lambda\}\\
\end{align*}
because $T_{-f_0^{**}}$ is $(w(X^{**}, \mathcal{P}(X)),w(X^{**}, \mathcal{P}(X)))$-continuous due to Lemma \ref{lm3.2} and $T_{-f_0^{**}}(0)=f_0^{**}$.

\medskip

Since $B$ is polynomial-star dense in $B^{**}$, we have that, if $\mathfrak{U}$ is a base of polynomial-star neighborhoods of $f_0^{**}$,

\begin{align*}
&\{ \lambda: \exists (g_{\alpha}^{**})\subset B^{**} \; | \; g_{\alpha}^{**}\xrightarrow{w(X^{**}, \mathcal{P}(X))}f_0^{**}, \tilde{\psi}(g_{\alpha}^{**})\to\lambda\}\\&=\underset{U \in \mathfrak{U}}{\cap} \overline{\tilde{\psi}(U\cap B^{**})}=\underset{U \in \mathfrak{U}}{\cap} \overline{\psi(U\cap B)}=Cl_B^{\mathcal{P}}(\psi,f_0^{**}),
\end{align*}
hence $\widehat{\psi}(M_{f_0^{**}}^{\mathcal{P}})\subset Cl_B^{\mathcal{P}}(\psi,f_0^{**})$ for all $\psi \in H^{\infty}(B_{C(K)})$, as needed.
\end{proof}

The reader can similarly check that for $X=C(K)$ and the algebra $A_{\infty}(B)$ of bounded holomorphic functions that can be continuously extended to the boundary, the polynomial cluster value problem reduces to the origin, points in the sphere of $C(K)$ and elements of $\overline{B}^{**}\setminus \overline{B}$.

\section{Strong peak points and the polynomial cluster value problem}

The discussions of this section concern solutions to the polynomial cluster value problem at special types of points. The results are mainly about the algebra $A_{\infty}(B)$, but we will also work with the other uniform algebras mentioned in the introduction. In turn we will obtain corresponding original cluster value theorems too. This is based on the observation that some points of $B$ admit a global function that distinguishes the point in a manner that we are about to describe broadly. 


\medskip

A peak point for a function space $H$ on a metric space $\Omega$ is an element $x$ of $\Omega$ for which there exists $f \in H$ such that $f(x)=1$ and $|f(y)|<1$ for all $y \in \Omega\setminus\{x\}$.

\medskip

Meanwhile, a strong peak point for a function space as before is an element $x$ of $\Omega$ for which there exists $f \in H$ satisfying $f(x)=1$ and that for all $\varepsilon>0$ we can find $\delta>0$ such that $d(x,y)>\varepsilon$ implies $|f(y)|<1-\delta$. In this case we say that $f$ peaks strongly at $x$.

\medskip

Examples of peak point sets for the uniform algebra $A(B)$ are the following, as exhibited in \cite{AL} and \cite{ACLP}.

\begin{theorem}[Acosta, Louren\c{c}o] If $K$ is separable, then all the extreme points in the ball of $X=C(K)$ are peak points for $A(B)$ as a function space on $\overline{B}$.
\end{theorem}

\begin{theorem}[Aron, Choi, Louren\c{c}o, Paques] All the extreme points in the ball of $X={\ell_{\infty}}$ are peak points for $A(B)$ as a function space on $\overline{B}$.
\end{theorem}

Examples and counterexamples of strong peak points are provided below, based on constructions in \cite{AL} and \cite{F}.

\begin{theorem}[Acosta, Louren\c{c}o]
If $K$ is any infinite compact Hausdorff space and $X=C(K)$, then there are no strong peak points for $A_{\infty}(B)$.
\end{theorem}

\begin{theorem} [Acosta, Louren\c{c}o] All the points in the unit sphere of $X=\ell_1$ are strong peak points for $A(B)$ as a function space on $\overline{B}$.
\end{theorem}

\begin{theorem}[Farmer] All the points in the unit sphere of a uniformly convex Banach space $X$ are strong peak points for $A(B)$ as a function space on $\bar{B}=\bar{B}^{**}$.
\end{theorem}

To obtain a local version of the previous result, let us consider the following notion: a Banach space $X$ is called locally uniformly convex at $x$ if $\lim\|y_n-x\|=0$ whenever $\{y_n\}\subset X$ is such that $\lim\|y_n\|=\|x\|$ and $\lim\|x+y_n\|=2\|x\|$. It is easy to check that uniformly convex spaces are locally uniformly convex.

\begin{theorem}
If $X$ is locally uniformly convex at a point $x$ of its sphere $S_X$ then $x$ is a strong peak point for $A(B)$ as a function space on $\bar{B}$.
\end{theorem}
\begin{proof}
Pick $x^*$ a norming functional for $x$. Since $X$ is locally uniformly convex at $x$, we have that for every $\varepsilon>0$ there is $\delta>0$ such that $\|y-x\|<\varepsilon$ whenever $y\in \overline{B}$ and $\|(x+y)/2\|>1-\delta$. Thus, if $y\in \overline{B}$ is such that $\text{Re}(x^*(y))>1-\delta$, then $\|\frac{1}{2}(x+y)\|\geq \frac{1}{2} \text{Re}(x^*(x+y))>1-\delta/2$, so consequently $\|y-x\|<\varepsilon$. By taking real and imaginary parts, it is easy to check that there is a number $\sigma>0$ such that, if $\text{Re}(x^*(y))\leq 1-\delta$ with $y\in \overline{B}$, then $|\frac{1}{2}(1+x^*)(y)|<1-\sigma$. Then $f=\frac{1}{2}(1+x^*) \in A(B)$ peaks strongly at $x$.
\end{proof}

A relationship between strong peak points and the cluster value problem is the next result that follows from material in Section 2 of \cite{ACGLM} .

\begin{proposition} 
Suppose that $x \in \bar{B}$ is a strong peak point for $A(B)$ as a function space on $\bar{B}$. Then for $H(B)$ between $A(B)$ and $A_{\infty}(B)$, the fiber $M_x$ reduces to one point, and thus the cluster value theorem for $H(B)$ at $x$ holds.
\end{proposition}

Let us present an analogous relationship between strong peak points and the polynomial cluster value problem.

\begin{theorem}
Suppose that $x \in \bar{B}$ is a strong peak point for $A_u(B)$ as a function space on $\bar{B}$. Then $x$ is a peak point for $A_u(B)$ as a function space on $\bar{B}^{**}$. Consequently for all $f \in H^{\infty}(B)$ that extend continuously to $B \cup \{x\}$ we have that the Gelfand transform of $f$ is constant on $M_x^{\mathcal{P}}$.
\end{theorem}
\begin{proof}
Since $x \in \bar{B}$ is a strong peak point for $A_u(B)$, there exists $g \in A_u(B)$ such that $g$ peaks strongly at $x$. 

\medskip

If $x_0 \in \bar{B}^{**}$ also satisfies $|\tilde{g}(x_0)|=1$, taking $(x_{\alpha})\subset B$ converging to $x_0$ in the polynomial-star topology, we get that $|g(x_{\alpha})|\to 1$. But this implies that $x_{\alpha}\to x$ in norm. So $x_0=x$. Hence $|\tilde{g}(y)|<1$ for all $y \in \bar{B}^{**}\setminus \{x\}$.

\medskip

Let $f \in H^{\infty}(B)$ extending continuously to $B \cup \{x\}$. Adding a constant to $f$, if necessary, we may assume that $f(y)\to 0$ as $y \to x$. Then $g^n f \to 0$ uniformly on $B$. Consequently $(\hat{g})^n \hat{f} \to 0$ uniformly on $M_{H^{\infty}(B)}$. Since $\hat{g}=1$ on $M_x^{\mathcal{P}}$, then $\hat{f}=0$ on $M_x^{\mathcal{P}}$.
\end{proof}

\begin{corollary}
Suppose that $x \in \bar{B}$ is a strong peak point for $A_u(B)$ as a function space on $\bar{B}$. Then for any uniform algebra $H(B)$ between $A_u(B)$ and $A_{\infty}(B)$, the polynomial fiber $M_x^{\mathcal{P}}$ reduces to one point, and thus the polynomial cluster value theorem for $H(B)$ at $x$ holds.
\end{corollary}

Let us finish with some specific (polynomial) cluster value theorems for $A_{\infty}(B)$ at special points using the aforementioned examples of strong peak points.

\begin{corollary}
For each $x \in S_{\ell_1}$, there is a (polynomial) cluster value theorem for $A_{\infty}(B_{\ell_1})$ at $x$.
\end{corollary}

\begin{corollary}
If $X$ is locally uniformly convex at $x \in S_X$, there is a (polynomial) cluster value theorem for $A_{\infty}(B)$ at $x$.
\end{corollary}

\section{Acknowledgement}

The authors thank Maite Fern\'andez Unzueta for numerous helpful discussions. We also wish to thank Sebasti\'an Lajara for drawing our attention to the notion of local uniform convexity. 

\bibliographystyle{amsplain}
\bibliography{mybibliography}

\begin{thebibliography}{10}

\bibitem{AL}
M.~D. Acosta and M.~L. Louren\c{c}o, \emph{Shilov boundary for holomorphic
  functions on some classical {B}anach spaces}, Studia Math. \textbf{179}
  (2007), no.~1, 27--39.

\bibitem{ACGLM}
R.~M. Aron, D.~Carando, T.~W. Gamelin, S.~Lassalle, and M.~Maestre,
  \emph{Cluster values of analytic functions on a {B}anach space}, Math. Ann.
  \textbf{353} (2012), 293--303.

\bibitem{ACLM}
R.~M.~Aron, D.~Carando, S.~Lassalle, and M.~Maestre, \emph{Cluster values of holomorphic functions of bounded type}, Trans.~Amer.~Math.~Soc. \textbf{368} (2016), no.~4, 2355--2369.

\bibitem{ACLP}
R.~M. Aron, Y.~S. Choi, M.~L. Louren\c{c}o, and O.~W. Paques, \emph{Boundaries
  for algebras of analytic functions on infinite dimensional {B}anach spaces},
  Contemp.~Math. \textbf{144} (1993), 15--22.

\bibitem{ACG}
R.~M. Aron, B.~J. Cole, and T.~W. Gamelin, \emph{Spectra of algebras of
  analytic functions on a {B}anach space}, J.~Reine Angew.~Math \textbf{415}
  (1991), 51--93.

\bibitem{AFGM}
R.~M.~Aron, J.~Falc\'o, D.~Garc\'ia, and M.~Maestre, \emph{Analytic structure in fibers}, Oberwolfach preprints.

\bibitem{AGGM}
R.~M.~Aron, P.~Galindo, D.~Garc\'ia, and M.~Maestre, \emph{Regularity and algebras of analytic functions in infinite dimensions}, Trans.~Amer.~Math.~Soc. \textbf{348} (1996), no.~2, 543--559.

\bibitem{B}
E.~Bishop, \emph{{U}niform {A}lgebras}, Proc.~Conf.~{C}omplex {A}nalysis
  ({M}inneapolis, 1964), Springer, Berlin, 1965, pp.~272--281.


\bibitem{CrGrMsSP}
D.~Carando, D.~Garc\'ia, M.~Maestre, and P.~Sevilla-Peris, \emph{On the spectra of algebras of analytic functions}, Contemp.~Math. \textbf{561} (2012) pp.~165--198.

\bibitem{C}
L.~Carleson, \emph{Interpolations by bounded analytic functions and the corona
  problem}, Ann.~of Math. \textbf{76} (1962), 547--559.

\bibitem{DG}
A.~M.~Davie and T.~W.~Gamelin, \emph{A theorem on polynomial-star approximation}, Proc.~Amer.~Math.~Soc. \textbf{106} (1989), no.~2, 351--356.

\bibitem{DKSTW}
Ronald~G. Douglas, Steven~G. Krantz, Eric~T. Sawyer, Sergei Treil, and Brett~D.
  Wick, \emph{The {C}orona {P}roblem; {C}onnections {B}etween {O}perator
  {T}heory, {F}unction {T}heory, and {G}eometry}, Springer, New York, 2014.

\bibitem{F}
J.~D. Farmer, \emph{Fibers over the sphere of a uniformly convex {B}anach
  space}, Michigan Math.~J. \textbf{45} (1998), no.~2, 211--226.

\bibitem{G-ua}
T.~W. Gamelin, \emph{{U}niform {A}lgebras}, Prentice-Hall, Englewood Cliffs,
  New Jersey, 1969.

\bibitem{G-afbs}
\bysame, \emph{{A}nalytic functions on {B}anach spaces}, Complex Potential
  Theory \textbf{439} (1994), 187--233.

\bibitem{G-itfa}
\bysame, \emph{Iversen's theorem and fiber algebras}, Pac.~J.~Math. \textbf{46} (1973), 389--414.

\bibitem{Hf} 
K.~Hoffman, \emph{Banach Spaces of Analytic Functions}, Prentice-Hall, Englewood Cliffs, New Jersey, 1962.

\bibitem{JO2}
W.~B. Johnson and S.~Ortega Castillo, \emph{The cluster value problem for
  {B}anach spaces}, Illinois J.~Math. \textbf{58} (2014), 405--412.

\bibitem{JO}
\bysame, \emph{The cluster value problem in spaces of continuous functions},
  Proc. Amer.~Math.~Soc. \textbf{143} (2015), 1559--1568.

\bibitem{Mc} 
G.~McDonald, \emph{The maximal ideal space of $H^{\infty}+C$ on the ball in $\mathbb{C}^n$}, Can.~Math.~J. \textbf{31} (1979), 79--86.

\bibitem{Mu}
J.~Mujica, \emph{Complex Analysis in Banach Spaces}, vol.~120, North-Holland Mathematics Studies, Amsterdam, 1986.

\bibitem{Mr}
Gerard~J. Murphy, \emph{{$C^*$}-algebras and {O}perator {T}heory}, Academic
  Press, Inc., Boston, 1990.

\bibitem{O} S.~Ortega Castillo, \emph{Cluster value problem in infinite-dimensional spaces}, Contemporary Mathematics 657 (2016), 165--178.

\bibitem{IJS}
I.~J. Schark, \emph{Maximal ideals in an algebra of bounded analytic
  functions}, Journal of Mathematics and Mechanics \textbf{10} (1961),
  735--746.

\bibitem{Ro} H.~Rossi, \emph{Holomorphically convex sets in several complex variables}, Ann.~of Math. \textbf{74} (1961), 470--493. 

\end{thebibliography}

\end{document}